\documentclass[12pt,a4paper]{amsart}

\usepackage[latin1]{inputenc}
\usepackage[spanish,english]{babel}
\usepackage{graphicx}

\usepackage{tabularx}

\usepackage{amsmath}
\usepackage{amsthm}
\usepackage{amsmath, amsthm, amscd, amsfonts, amssymb, graphicx, color}

\usepackage{pifont}

\usepackage{amsfonts}
\usepackage{graphicx}

\newtheorem{exer}{\sffamily\bfseries Ejercicio}
\newtheorem{teo}{Theorem}

\newtheorem{cor}[teo]{Corollary}
\newtheorem{prop}[teo]{Proposition}

\newtheorem{rem}[teo]{Remark}

\numberwithin{equation}{section}

\newcommand{\kah}{\mathcal{K}(\mathcal{H})^{ah}}
\newcommand{\kh}{\mathcal{K}(\mathcal{H})^{h}}
\newcommand{\bh}{\mathcal{B}(\mathcal{H})^h}
\newcommand{\oa}{\mathcal{O}_A}

\newcommand{\aaa}{\mathcal{A}}
\newcommand{\bb}{\mathcal{B}}
\newcommand{\kk}{\mathcal{K}}

\newcommand{\D}{\mathcal D}

\newcommand{\N}{\mathbb N}
\newcommand{\C}{\mathbb C}

\newcommand{\R}{\mathbb R}

\newcommand{\bit}{\begin{itemize}}
\newcommand{\eit}{\end{itemize}}
\newcommand{\be}{\begin{enumerate}}
\newcommand{\ee}{\end{enumerate}}
\newcommand{\bx}[1]{\begin{exer}\rm{#1}}
\newcommand{\ex}{\end{exer}}

\newcommand{\ba}{\begin{array}}
\newcommand{\ea}{\end{array}}
\newcommand{\bc}{\begin{center}}
\newcommand{\ec}{\end{center}}

\newcommand{\g}{\gamma}

\newcommand{\hh}{\mathcal{H}}

\newcommand{\bq}{\begin{equation}}
\newcommand{\eq}{\end{equation}}

\begin{document}

\title{\vspace*{0cm}BEST APPROXIMATION BY DIAGONAL COMPACT OPERATORS
}
\date{\today}
\author{Tamara Bottazzi $^1$ and Alejandro Varela$^{1,2}$}

\address{$^1$ Instituto Argentino
de Matemática ``Alberto P. Calder\'on'', Saavedra 15 3º piso,
(C1083ACA) Ciudad Aut\'onoma de Buenos Aires, Argentina} 

\address{$^2$ Instituto de Ciencias, Universidad Nacional de General Sarmiento, J.
M. Gutierrez 1150, (B1613GSX) Los Polvorines, Pcia. de Buenos Aires, Argentina} 
\email{tpbottaz@ungs.edu.ar, avarela@ungs.edu.ar}


\begin{abstract} We study the existence and characterization properties of compact Hermitian operators $C$ on a separable Hilbert space $\hh$ such that
$$
\left\|C\right\|\leq \left\|C+D\right\|, \text{ for all } D\in\D(\kh)
$$
\text{ or equivalently }
$$
\left\|C\right\|=\min_{D\in\D(\kh)}\left\|C+D\right\|=\text{dist}\left(C,\D(\kh)\right)
$$
where $\D(\kh)$ denotes the space of compact real diagonal operators in a fixed base of $\hh$ and $\left\|.\right\|$ is the operator norm. We also exhibit a positive trace class operator that fails to attain the minimum in a compact diagonal.
\end{abstract}

\maketitle

\section{Introduction} \label{intro}
Let $\hh$ be a separable Hilbert space, $\kk(\hh)$ be the algebra of compact operators and $\D(\kh)$, the $C^*$ subalgebra of real diagonal compact operators (with the canonical base or any other fixed base). In this paper we study the existence and describe Hermitian compact operators $C$ such that
$$\left\|C\right\|\leq \left\|C+D\right\|,\  \rm{for\ all}\ D\in \D(\kh),$$
or equivalently
$$\left\|C\right\|={\rm dist}(C,\D(\kh)).$$
Where $\left\|.\right\|$ denotes the usual operator norm. These operators $C$ will be called minimal. Our interest in them comes from the study of minimal length curves of the orbit manifold of a self-adjoint compact operator $A$ by a particular unitary group (see \cite{andruchow larotonda}), that is
$$
\mathcal{O}_A=\{uAu^*: u \text{ unitary in } B(H) \text{ and } (u-1)\in\mathcal{K}(H)\}.
$$ 
The tangent space for any $b\in\mathcal{O}_A$ is 
$$T_b(\oa)=\{zb-bz: z \in\kah \}.$$
Where the suffix $ah$ refers to the anti-Hermitian operators (analogously, the suffix $h$ refers to Hermitian operators). If $x\in T_b(\oa)$, the existence of a (not necessarily unique) minimal element $z_0$ such that
$$\left\|x\right\|=\|z_0\|=\inf \left\{\left\|z\right\|:\ z \in \kah, \ zb-bz=x \right\}$$
allows the description of minimal length curves of the manifold by the parametrization
$$\gamma(t)=e^{tz_0}\  b\  e^{-tz_0} ,\ t\in[-1,1].$$
These $z_0$ can be described as $i(C+D)$, with $C\in\kh$ and $D$ a real diagonal operator in the orthonormal base of eigenvectors of $A$.

If we consider a von Neumann algebra $\mathcal{A}$ and a von Neumann subalgebra, named $\mathcal{B}$, of $\mathcal{A}$, it has been proved in \cite{dmr1} that for each $a\in\mathcal{A}$ there always exists a minimal element $b_0$ in $\mathcal{B}$. It means that $\left\|a+b_0\right\|\leq \left\|a+b\right\|$, for all $b\in \mathcal{B}$. For example, if $M_n^h(\C)$ is the algebra of Hermitian matrices of $n\times n$ and $\D(M_n^h(\C))$ is the subalgebra of diagonal Hermitian matrices (or diagonal real matrices), it is easy to prove that, for every $M\in M_n^h(\C)$ there always exists a minimal element $D\in \D(M_n^h(\C))$.

However, in the case of $\kh$, which is only a $C^*$-algebra, the existence of a best approximant in the general case is not guaranteed. In the particular case that $C\in \kh$ has finite rank, it was proved in Proposition 5.1 in \cite{andruchow larotonda} that there exists a minimal compact diagonal element. 

The results we present in this paper are divided in two parts. In the first one we describe a particular case of minimal operators that allow us to prove there is not always a minimal diagonal compact operator. In the second part we present properties and characterizations of minimal compact operators in general. 

\section{Preliminaries and notation} \label{preliminares}

Let $(\hh,\left\langle ,\right\rangle)$ be a separable Hilbert space with the norm $\left\|x\right\|=\left\langle x,x\right\rangle^{1/2}$, for each $x\in \hh$. We denote with $\kk(\hh)$, the two-sided closed ideal of compact operators on $\hh$, with $\bb_1(\hh)$, the space of trace class operators, and $\bb(\hh)$ the set of bounded operators.

We denote with $\left\|T\right\|$ the usual operator norm of $T\in \bb(\hh)$ and $\left\|L\right\|_1=\text{tr}(\left|L\right|)=\text{tr}\left[(L^*L)^{1/2}\right]$, the trace norm of $L\in \bb_1(\hh)$. It should cause no confusion the use of the same notation $\left\|.\right\|$ to refer to the operator norm or a norm on $\hh$, it should be clear from the context.

If $\aaa$ is any of the previous sets, we denote with $\D(\aaa)$ the set of diagonal operators, that is
$$\D(\aaa)=\left\{T\in \aaa:\ \left\langle Te_i,e_j\right\rangle=0\ ,\ \text{ for all } i\neq j\right\},$$
where $\left\{e_k\right\}_{k=1}^{\infty}$ is the canonical (or any other fixed) base of $\hh$. We consider an operator $T\in \bb(\hh)$ like an infinite matrix defined for each $i,j\in \N$ as $T_{ij}=\left\langle Te_i,e_j\right\rangle$. In this sense, the $j$th-column and $i$th-row of $T$ are the vectors in $l^2$ given by $c_j(T)=\left(T_{1j},T_{2j},...\right)$ and $f_j(T)=\left(T_{i1},T_{i2},...\right)$, respectively.

Let $L\in \bh$, we denote the positive and negative parts of $L$ as:
$$L^+=\dfrac{\left|L\right|+L}{2}\ \ {\rm and} \ \ L^-=\dfrac{\left|L\right|-L}{2},$$
respectively. 

We use $\sigma(T)$ and $R(T)$ to denote the spectrum and range of $T\in \bh$, respectively. 

We define $\Phi:\bb(\hh)\to \D(\bb(\hh))$, $\Phi(X)= diag (X)$, which essentially takes the main diagonal (i.e the elements of the form $\left\langle Xe_i,e_i\right\rangle_{i\in\N}$) of an operator $X$ and builds a diagonal operator in the canonical base or the chosen fixed base of $\hh$. For a given sequence $\{d_n\}_{n\in\N}$ we denote with Diag$\big( \{d_n\}_{n\in\N}\big)$ the diagonal (infinite) matrix with $\{d_n\}_{n\in\N}$ in its diagonal  and $0$ elsewhere.

We define the space $\kh/\D(\kh)$ with the usual quotient norm 
$$\left\|\left[C\right]\right\|=\inf_{D\in \\\D(\kh)}\ \left\|C+D\right\|=\text{dist}(C,\\\D(\kh))$$
for each class $[C]=\left\{C+D:\ D\in \\\D\left(\kh\right)\right\}$. 

Given an operator $C\in \kh$, if there exists an operator $D_1$ compact and diagonal such that
$$\left\|C+D_1\right\|={\rm dist}\left(C,\D\left(\kh\right)\right),$$
we say that $D_1$ is a best approximant of $C$ in $\D(\kh)$. In other terms, the operator $C+D_1$ verifies the following inequality
$$\left\|C+D_1\right\|\leq \left\|C+D\right\|$$
for all $D\in \D(\kh)$. In this sense, we call $C+D_1$ a \textbf{minimal} operator or similarly we say that $D_1$ is minimal for $C$.

\section{The existence problem of the best approximant}
Some examples of compact Hermitian operators that possess a closest compact diagonal are: i) those constructed with Hermitian square matrices in their main diagonal, ii) tridiagonal operators with zero diagonal, and iii) finite rank compact operators (see \cite{andruchow larotonda} for a proof).

In the rest of this section we study some examples of compact Hermitian operators with a unique best diagonal approximant. Then, we use this example to show an operator which has no best compact diagonal approximant. We use frequently the fact that any bounded operator $T$ can be described uniquely as an infinite matrix with the notation $T_{ij}$ that we introduced in Section \ref{preliminares} using the canonical (or any other fixed) base.

The following statement is about a set of compact symmetric operators ($L=L^t$), which has the following property: every operator has a column (or row) such that every different column (or row) is orthogonal to it (considering $L$ as an infinite matrix). This result has its origins in the finite dimensional result obtained in \cite{kv}.

\begin{teo} \label{caso3}
Let $T\in \kh$ described as an infinite matrix by $\left(T_{ij}\right)_{i,j\in \N}$.
Suppose that $T$ satisfies:
\be
\item $T_{ij}\in\R$ for each $i,j\in \N$,
\item there exists $i_0\in N$ satisfying $T_{i_0 i_0}=0$, with $T_{i_0 n}\neq 0$, for all $n\neq i_0$,
\item if $T^{[i_0]}$ is the operator $T$ with zero in its $i_0$th-column and $i_0$th-row then $$
\left\|c_{i_0}(T)\right\|\geq \left\|T^{[i_0]}\right\|
$$
(where $\left\|c_{i_0}(T)\right\|$ denotes the Hilbert norm of the $i_0$th-column of $T$), and
\item if the $T_{nn}$'s satisfy that, for each $n\in\N$, $n\neq i_0$:
$$
T_{nn}=-\dfrac{\left\langle c_{i_0}(T),c_n(T)\right\rangle}{T_{i_0n}}.
$$
\ee
then $T$ is minimal, that is
$$
\left\|T\right\|=\left\|c_{i_0}(T)\right\|=\inf_{D\in \D(\kh)}\left\|T+D\right\|
$$
and moreover, $D=\text{Diag}((T_{nn})_{n\in\N})$ is the unique bounded minimal diagonal operator for $T$.  
\end{teo} 

\begin{proof}
Without loss of generality we can suppose that $T$ is a compact operator with real entries and $i_0=1$, therefore it has the matrix form given by
$$T= \begin{pmatrix}
0&T_{12}&T_{13}&T_{14}&\cdots\\
T_{12}&T_{22}&T_{23}&T_{24}&\cdots\\
T_{13}&T_{23}&T_{33}&T_{34}&\cdots\\
T_{14}&T_{24}&T_{34}&T_{44}&\cdots\\
\vdots&\vdots&\vdots&\vdots&\ddots\\
\end{pmatrix}.
$$
The hypothesis in this case are
\bit
\item $i_0=1$ with $T_{1n}\neq 0$, $\forall n\in\N-\left\{1\right\}$.
\item $\left\|c_1(T)\right\|\geq \left\|\underbrace{\begin{pmatrix}
0&0&0&0&\cdots\\
0&T_{22}&T_{23}&T_{24}&\cdots\\
0&T_{23}&T_{33}&T_{34}&\cdots\\
0&T_{24}&T_{34}&T_{44}&\cdots\\
\vdots&\vdots&\vdots&\vdots&\ddots\\
\end{pmatrix}}_{=T^{[1]}}\right\|=\left\|T^{[1]}\right\|$.
\item Each $T_{nn}$ fulfills:
$$T_{nn}=-\dfrac{\left\langle c_1(T),c_n(T)\right\rangle}{T_{1n}}\ \ \ {\rm for\: every}\  n\in \N-\left\{1\right\}.$$ 
\eit

There are some remarks to be made:

\be
\item First note that for every $i\in \N$
$$\left|T_{ii}\right|=\left|\left\langle T^{[1]}e_i,e_i\right\rangle\right|\leq \left\|T^{[1]}e_i\right\|\left\|e_i\right\|\leq \left\|T^{[1]}\right\|\leq \left\|c_1(T)\right\|<\infty$$
namely, $(T_{ii})_{i\in\N}$ is a bounded sequence (each $T_{ii}$ is a diagonal element of $T^{[1]}$ in the canonical or fixed base). 
\item A direct computation proves that $\left\|c_1(T)\right\|$ and  $-\left\|c_1(T)\right\|$ are eigenvalues of $T$ with 
$$v_+=\dfrac{1}{\sqrt{2}\left\|c_1(T)\right\|}\left(\left\|c_1(T)\right\|e_1+c_1(T)\right)\ \ \ \rm{and}\ \ \ v_-=\dfrac{1}{\sqrt{2}\left\|c_1(T)\right\|}\left(\left\|c_1(T)\right\|e_1-c_1(T)\right),$$
which are eigenvectors of $\left\|c_1(T)\right\|$ and $-\left\|c_1(T)\right\|$, respectively. Let us consider the space $V=Gen\left\{v_+,v_-\right\}$:
\bit
\item If $w\in V$, then $\left\|Tw\right\|^2=\left\|c_1(T)\right\|^2\left\|w\right\|^2$.
\item If $y\in V^{\perp}$, then $\left\|Ty\right\|=\left\|T^{[1]}y\right\|\leq \left\|T^{[1]}\right\|\left\|y\right\|$.
\eit

Then, for every $x=w+y\in \hh$, with $w\in V$ and $y\in V^{\perp}$:
$$\left\|T(w+y)\right\|^2=\left\|Tw\right\|^2+\left\|Ty\right\|^2\leq\left\|c_1(T)\right\|^2\left\|w\right\|^2+\left\|T^{[1]}\right\|^2\left\|y_1\right\|^2\leq\left\|c_1(T)\right\|^2\left\|x\right\|^2$$

Therefore,
$$\left\|T\right\|=\left\|c_1(T)\right\|.$$

\item Let $D'\in \D(\kh)$ and define $(\underbrace{T+D'}_{=T'})e_i=T'(e_i)=c_i(T')$ for each $i\in \N$, then the following properties are satisfied:
\bit
\item If $D'_{11}\neq 0$ then
$$\left\|T'(e_1)\right\|^2=\left\|c_1(T')\right\|^2=D'_{11}+\left\|c_1(T)\right\|^2>\left\|c_1(T)\right\|^2=\left\|T\right\|^2\Rightarrow \left\|T'\right\|>\left\|T\right\|.$$
Therefore, we can assume that if $T+D'$ is minimal then $D'_{11}=0$.
\item Now suppose that there exists $i\in \N$, $i>1$, such that $D'$ does not have its $i$th-column orthogonal to the first one, that is:
$$\left\langle T'e_1,T'e_i\right\rangle=\left\langle c_1(T'),c_i(T')\right\rangle=a\neq 0.$$
Then,
$$T'\left(\dfrac{c_1(T)}{\left\|c_1(T)\right\|}\right)=\left(\left\|c_1(T)\right\|,\dfrac{a_2}{\left\|c_1(T)\right\|},\ldots,\dfrac{a_i}{\left\|c_1(T)\right\|},\ldots\right)\Rightarrow\left\|T'(c_1(T))\right\|^2>\left\|c_1(T)\right\|^2=\left\|T\right\|.$$
Hence, $\left\|T'\right\|>\left\|T\right\|$.
\eit
Therefore, $D=\text{Diag}((T_{nn})_{n\in\N})$ is the unique minimal diagonal for $T$ and it is bounded. 
\ee

\end{proof}

Note that the minimal diagonal obtained in Theorem \ref{caso3} is clearly bounded but we do not know if it is compact. An interesting question is if there exist an operator $T$ which fulfills the hypothesis of Theorem \ref{caso3} and it has an only minimal bounded diagonal non compact. To answer this question we analyzed several examples, we show the most relevant among them.

Let $\g\in \R$ be such that $\left|\g\right|<1$ and take an operator $T\in \bh$ defined as $(T_{ij})_{i,j\in\N}$ where
$$T_{ij}=\left\{
\begin{array}{c l r l}
0& \mbox{ if } & i=j\\
\g^{\max\left\{i,j\right\}-2}& \mbox{ if } &i\neq j&\mbox{and \ } j,i\neq 1\\
\g^{\left|i-j\right|}& \mbox{ if } &j=1&\mbox{or  \ }i=1\\
\end{array}
\right.$$
Writing $T$ as an infinite matrix
$$T=\begin{pmatrix}
0&\g&\g^2&\g^3&\g^4&\cdots\\
\g&0&\g&\g^2&\g^3&\cdots\\
\g^2&\g&0&\g^2&\g^3&\cdots\\
\g^3&\g^2&\g^2&0&\g^3&\cdots\\
\g^4&\g^3&\g^3&\g^3&0&\cdots\\
\vdots&\vdots&\vdots&\vdots&\vdots&\ddots\\

\end{pmatrix}.
$$
$T$ is symmetric and $c_n(T)$ is the the $n$th-column. Then, direct calculations show that 
$$\text{tr}(T^*T)=\text{tr}(T^2)=\sum_{n=1}^{\infty}(T^*T)_{nn}=\sum_{n=1}^{\infty}\left\langle c_n(T),c_n(T)\right\rangle=\dfrac{-1+4 \g^2 + 2 \g^4-4 \g^6+\g^8}{\g^2 (-1 +\g^2)^2}<\infty .$$
Then, $T$ is a Hilbert-Schmidt operator. Consider a diagonal operator $D$, given by $D=\text{Diag}\left((d_n)_{n\in\N}\right)$, with the sequence $(d_n)_{n\in\N}\subset \R$ such that
\be
\item $d_{1}=0$.
\item $\left\langle c_1(T),c_n(T+D)\right\rangle=0$, for every $n\in \N$, $n>1$.
\ee

Indeed, for every $n>3$ each $d_n$ is uniquely determined by
$$
d_n=-\dfrac{\g^2 -\g^n}{(1-\g) \g^2}+\dfrac{\g^n}{-1+\g^2}.
$$
We can also note that $d_n\to \dfrac{1}{\g-1}$ when $n\to \infty$, so the diagonal operator $D=\text{Diag}\left((d_n)_{n\in\N}\right)$ is bounded but non compact.

On the other hand, if we consider $T^{[1]}$, the operator given by

$$T^{[1]}=\begin{pmatrix}
0&0&0&0&0&\cdots\\
0&0&\g&\g^2&\g^3&\cdots\\
0&\g&0&\g^2&\g^3&\cdots\\
0&\g^2&\g^2&0&\g^3&\cdots\\
0&\g^3&\g^3&\g^3&0&\cdots\\
\vdots&\vdots&\vdots&\vdots&\vdots&\ddots\\

\end{pmatrix},$$

then $T^{[1]}$ is also a Hilbert-Schmidt operator. Then $T^{[1]}+D\in \bb(\hh)$. Now consider the operator $T_r$, given by
\begin{equation}\label{contraejemplo}
T_r=\begin{pmatrix}
0&r\g&r\g^2&r\g^3&r\g^4&\cdots\\
r\g&0&\g&\g^2&\g^3&\cdots\\
r\g^2&\g&0&\g^2&\g^3&\cdots\\
r\g^3&\g^2&\g^2&0&\g^3&\cdots\\
r\g^4&\g^3&\g^3&\g^3&0&\cdots\\
\vdots&\vdots&\vdots&\vdots&\vdots&\ddots\\
\end{pmatrix} 
\end{equation}
with $r=\dfrac{\left\|T^{[1]}+D\right\|}{\left\|c_1(T)\right\|}$. Then, we claim that the following operator
$$T_r+D=\begin{pmatrix}
0&r\g&r\g^2&r\g^3&r\g^4&\cdots\\
r\g&T_{22}&\g&\g^2&\g^3&\cdots\\
r\g^2&\g&T_{33}&\g^2&\g^3&\cdots\\
r\g^3&\g^2&\g^2&T_{44}&\g^3&\cdots\\
r\g^4&\g^3&\g^3&\g^3&d_{55}&\cdots\\
\vdots&\vdots&\vdots&\vdots&\vdots&\ddots\\

\end{pmatrix}
$$
is minimal and unique, which means:
$$\left\|\left[T_r\right]\right\|=\inf_{D'\in \D(\bb_h(H))}\ \left\|T+D'\right\|=\inf_{D'\in \D(\kh)}\ \left\|T+D'\right\|=\left\|T_r+D\right\|$$
This is true because $T_r$ is an operator which clearly satisfies the hypothesis of Theorem \ref{caso3}. It follows from the non-compacity of $D$ that there is no best compact diagonal approximation of $T_r$.

The operator $T_r$ is also a positive trace class operator. In effect, if we consider the lower triangular operator $C_a\in \bb(\hh)$, given by $(C_a)_{ij}=a^i$, for $i\geq j$, and take $a=\sqrt{\g}$, then
$$C_{\sqrt{\g}}^*C_{\sqrt{\g}}=\dfrac{1}{1-\g}
\begin{pmatrix}
\g&\g^2&\g^3&\g^4&\g^5&\cdots\\
\g^2&\g^2&\g^3&\g^4&\g^5&\cdots\\
\g^3&\g^3&\g^3&\g^4&\g^5&\cdots\\
\g^4&\g^4&\g^4&\g^4&\g^5&\cdots\\
\vdots&\vdots&\vdots&\vdots&\vdots&\ddots\\

\end{pmatrix}=\dfrac{1}{1-\g}\ Q.$$

Therefore, 
$$\text{tr}(\left|Q\right|)=(1-\g)\ \text{tr}\left(\left|C_{\sqrt{\g}}^*C_{\sqrt{\g}}\right|\right)=(1-\g)\ \text{tr}\left(C_{\sqrt{\g}}^*C_{\sqrt{\g}}\right)=\text{tr}(Q),$$

which shows that $Q\in \bb_1(\hh)$. On the other hand, the operator
$$R=\begin{pmatrix}
0&r\g&r\g^2&r\g^3&\cdots\\
r\g&0&0&0&\cdots\\
r\g^2&0&0&0&\cdots\\
r\g^3&0&0&0&\cdots\\
\vdots&\vdots&\vdots&\vdots&\ddots\\
\end{pmatrix}$$

has finite rank, thus $\begin{pmatrix}
0&\cdots\\
\vdots&Q
\end{pmatrix}+R\in \bb_1(\hh)$. But also $\begin{pmatrix}
0&\cdots\\
\vdots&Q
\end{pmatrix}+R-\text{diag}(Q)=T_r$, which is equivalent to say that $\begin{pmatrix}
0&\cdots\\
\vdots&Q
\end{pmatrix}+R$ is in the same class that $T_r$. As $\text{diag}(Q)\in \bb_1(\hh)$, it follows that $T_r\in \bb_1(\hh)$. Moreover, since $Q$ and $R$ are positive then $T_r$ is also positive.

\begin{rem}[About the implications of the uniqueness condition on the existence of minimal diagonal operators]
For a given Hermitian compact operator $C$ the existence of a unique bounded real diagonal operator $D_0$ minimal for $C$ does not imply that $D_0$ is not compact. On the other hand, if there exist infinite bounded real diagonal operators that are minimal for $C$, this does not imply that there exists a compact minimal diagonal.

The next examples of operators show that the existence of a unique (respectively non unique) minimal diagonal does not necessarily imply that there does not exist (respectively that there exists) a minimal compact diagonal.  
\be
\item Let $L\in \D(\kh)$, $L\neq 0$, then $-L$ is the only minimal diagonal compact operator. In this case, we can observe that there is uniqueness for the minimal, but the best approximant is also compact.
\item Let us consider the example $T_r$ defined in (\ref{contraejemplo}) and the block operator $S=\begin{pmatrix}
S_n&0\\
0&T_r
\end{pmatrix}$, where $S_n\in M_n^h(\C)$ is a matrix whose quotient norm is $\left\|[T_r]\right\|$ and has infinite minimal diagonals of $n\times n$ (consider matrices like those in \cite{ammrv2}, \cite{ammrv3} or \cite{kv}). Then, all minimal diagonal bounded operators for $S$ are of the form $D'=\begin{pmatrix}
D_n&0\\
0&D
\end{pmatrix}$, with any of the infinite $D_n$ minimals for $S_n$ and $D$ the unique minimal bounded diagonal operator for $T_r$. Thus, none of these $D'$ is compact. This case shows that if uniqueness of a minimal diagonal does not hold this does not necessarily imply the existence of a minimal compact diagonal operator.
\ee
\end{rem}

\section{A characterization of minimal compact operators}

In the previous section we showed an example of a compact operator $T_r$ that has no compact diagonal best approximant. The main property that allowed us to prove the non existence of a minimal compact diagonal is the uniqueness of the best approximant for $T_r$.  

Nevertheless, there are a lot of compact operators which have at least one best compact diagonal approximation, for example the operators of finite rank. The spirit of this part follows the main ideas in \cite{varela}. The main purpose of this subsection is to study properties and equivalences that characterize minimal compact operators.

The next two Propositions are closely related with the Hahn-Banach theorem for Banach spaces and they relate the space $\kh$ with $\bb_1(\hh)^h$.

\begin{prop}\label{minimo Y0}
Let $C\in \kh$ and consider the set
$$\mathcal{N}=\left\{Y\in \bb_1(\hh)^h:\ \left\|Y\right\|_1=1,\ \text{tr}(YD)=0\ , \forall \ D\in \D(\kh)\right\}.$$
Then, there exists $Y_0\in \mathcal{N}$ such that
\begin{equation}
\left\|[C]\right\|=\inf_{D\in \D(\kh)}\ \left\|C+D\right\|=\text{tr}(Y_0C).
\end{equation}

\end{prop}

\begin{proof}

It is an immediate consequence from the Hahn-Banach theorem that since $\D(\kh)$ is a closed subspace of $\kh$ and $C\in \kh$, then there exists a functional $\rho:\kh\to \R$ such that $\left\|\rho\right\|=1$, $\rho(D)=0$, $\forall D\in \D(\kh)$, and
$$
\rho(C)=\inf_{D\in \D(\kh)}\ \left\|C+D\right\|=\text{dist}(C,\D(\kh)).
$$
But, since any functional $\rho$ can be written as $\rho(.)=\text{tr}(Y_0.)$, with $Y_0\in \bb_1(\hh)$, the result follows.
\end{proof}

\begin{prop}[Banach Duality Formula] \label{duality}
Let $C\in \kk(H)$, then
\begin{equation}
\left\|[C]\right\|=\inf_{D\in \D(\kk(H))}\ \left\|C+D\right\|=\max_{Y\in \mathcal{N}}\ \left|\text{tr}(CY)\right|.
\end{equation}
\end{prop}

\begin{proof}
%
%
%
Let $C\in \kk(\hh)$. By Proposition \ref{minimo Y0}, there exists $Y_0\in \mathcal{N}$ such that
$$\inf_{D\in \D(\kh)}\ \left\|C+D\right\|=\text{tr}(Y_0C).$$
Then 
$$\inf_{D\in \D(\kh)}\ \left\|C+D\right\|=\text{tr}(Y_0C)\leq \max_{Y\in \mathcal{N}}\ \left|\text{tr}(CY)\right|.$$
On the other side, consider for each $D\in \D(\kh)$ the set 
$$\mathcal{N}_D=\left\{Y\in \bb_1(\hh)^h:\ \left\|Y\right\|_1=1,\ \text{tr}(YD)=0\right\},$$
if we fix $D\in \D(\kk(H))$, we have 
$$\sup_{Y\in \mathcal{N}_D}\ \left|\text{tr}(YC)\right|=\sup_{Y\in \mathcal{N}_D}\ \left|\text{tr}\left(Y(C+D)\right)\right|.$$
Take the functional $\varphi: \bb_1(\hh)^h\longrightarrow \R$, defined by $\varphi(Y)=\text{tr}\left(Y(C+D)\right)$. We have that 
$$\left\|\varphi\right\|=\left\|C+D\right\|$$
Therefore,
$$\sup_{Y\in \mathcal{N}_D}\ \left|\text{tr}(YC)\right|=\sup_{Y\in \mathcal{N}_D}\ \left|\text{tr}\left(Y(C+D)\right)\right|\leq \left\|C+D\right\|,$$
for each fixed compact diagonal operator $D$. \\
Then  $\mathcal {N}\subseteq \mathcal{N}_D$ for all $D\in \D(\kh)$. Hence
$$\sup_{Y\in \mathcal{N}}\ \left|\text{tr}(YC)\right|\leq \sup_{Y\in \mathcal{N}_D}\ \left|\text{tr}(YC)\right|\leq \left\|C+D\right\|.$$
\end{proof}

Note that the annihilator of $\D(\kh)$ (i.e, $Y\in \bb_1(\hh)$ such that $\text{tr}(YD)=0$ for every $D\in \D(\kh)$) and the annihilator of $\D(\bh)$ are the same set. The proof of this fact is a direct consequence of the definition and we omit it. Moreover, it is easy to prove that if $Y\in \D(\bh)^{\perp}$, then $\text{Diag}(Y)=0$.

It is trivial that
$$\inf_{D\in \D(\bh)}\ \left\|C+D\right\|\leq \inf_{D\in \D(\kh)}\ \left\|C+D\right\|$$
Observe that there always exists $D_0\in \D(\bh)$ such that $\left\|C+D_0\right\|=\inf_{D\in \D(\bb(H))^h}\ \left\|C+D\right\|$, since $\bb(\hh)$ is a von Neumann algebra and $\D(\bb(\hh))$ is a von Neumann subalgebra of $\bb(\hh)$ (see \cite{dmr1}).

With the above properties we can prove the reverse inequality, as we show in the following proposition. 

\begin{prop} \label{igualdad de infimos}
Let $C\in \kh$, then
$$\inf_{D\in \D(\bh)}\ \left\|C+D\right\|=\inf_{D\in \D(\kh)}\ \left\|C+D\right\|.$$
\end{prop}

\begin{proof}
Let $D_0$ a minimal bounded diagonal operator such that
$$\inf_{D\in \D(\bb(H))^h}\ \left\|C+D\right\|=\left\|C+D_0\right\|.$$

Then, using Proposition \ref{minimo Y0}, there exists $Y_0\in \bb_1(\hh)$ such that
$$
\inf_{D\in \D(\kk(H))^h}\ \left\|C+D\right\|=\left|\text{tr}(Y_0C)\right|=\left|\text{tr}(Y_0(C+D_0))\right|\leq \left\|C+D_0\right\|
$$
which completes the proof.

\end{proof}{}

A natural fact that has been proved for minimal Hermitian matrices is a  balanced spectrum property: if $M\in M_n^h(\C)$ and $M$ is minimal then $\left\|M\right\|$ and $-\left\|M\right\|$ are in the spectrum of M. This property holds for minimal compact operators.

\begin{prop}[Balanced spectrum property] \label{espectro centrado}
Let $C\in \kk(H)^h$, $C\neq 0$. Suppose that there exists $D_1\in \D(\kh)$ such that $C+D_1$ is minimal, then
$$
\pm\|C+D_1\|\in\sigma(C+D_1).
$$
\end{prop}

\begin{proof}
The proof is a routine application of functional calculus to the Hermitian operator $C+D_1$.
%
%
%
%
%
%
%

\end{proof}

\begin{teo} \label{teoprincipal}
Let $C \in \kh$ and $D_1\in \D(\kh)$. Consider $E_+$ and $E_-$, the spectral projections of the eigenvalues $\lambda_{max}(C+D_1)$ and $\lambda_{min}(C+D_1)$, respectively. The following statements are equivalent:
\begin{enumerate}
\item $C+D_1$ is minimal.
\item There exists $X\in \bb_1(\hh)$, $X\neq 0$, such that 
\begin{itemize}
\item $\left\langle Xe_i,e_i\right\rangle=0\ ,\ \forall i\in \mathbb{N};$
\item $\left|\text{tr}(X(C+D_1))\right|=\left\|C+D_1\right\|\left\|X\right\|_1;$
\item $E_+X^+=X^+\ ,\ E_-X^-=X^-.$
\end{itemize}
\item $\lambda_{min}(C+D_1)+\lambda_{max}(C+D_1)=0$ and for each $D\in \D(\kh)$ there exists $y\in R(E_+)\ ,\ z\in R(E_-)$ such that:
\begin{itemize}
\item $\left\|y\right\|=\left\|z\right\|=1;$
\item $\left\langle Dy,y\right\rangle \leq \left\langle Dz,z\right\rangle.$
\end{itemize}
\end{enumerate}\end{teo}

\begin{proof}

(2) $\Rightarrow$ (1) Let $C\in \kh$ and $D_1\in \D(\kh)$. If there exists $X \in \bb_1(\hh)^h$ which fulfills the properties in \textit{2}, then:
$$\left\|C+D_1\right\|=\dfrac{\text{tr}(X(C+D_1))}{\left\|X\right\|_1}=\text{tr}\left(\dfrac{X}{\left\|X\right\|_1}C\right)\leq \sup_{Y\in \mathcal{N}}\ \left|\text{tr}(YC)\right|=\inf_{D\in \D(\kh)}\ \left\|C+D\right\|,$$
where the last equality holds for the Banach Duality Formula (see Proposition \ref{duality}). Then, $C+D_1$ is minimal.

(1) $\Rightarrow$ (2) Without loss of generality, we can suppose that $\left\|C+D_1\right\|=1$.
The proof of this part follows the same techniques used in Theorem 2 in \cite{varela} for matrices and we include it for the sake of completeness. For Banach duality formula there exists $X\in \bb_1(H)^h$ such that
$$\left\langle Xe_i,e_i\right\rangle=0\ ,\ \forall i\in \mathbb{N}\ ,\ \left\|X\right\|_1=1\ ,\ \text{tr}(X(C+D_1))=\text{tr}(XC)=1.$$
Let us prove that $X(C+D_1)=(C+D_1)X$. Since $C+D_1$ is minimal Proposition \ref{espectro centrado} implies that $-1,1\in \sigma(C+D_1)$. Consider the spectral projections $E_+$, $E_-$ and $E_3=I-E_+-E_-$. The operators $C+D_1$ and $X$ can be written matricially, in therms of the orthogonal decomposition $\hh=R(E_+)\oplus R(E_-)\oplus R(E_3)$, as
$$C+D_1=\begin{pmatrix}
I&0&0\\
0&-I&0\\
0&0&(C+D_1)_{3,3}\\
\end{pmatrix} \ {\rm and}\ X=\begin{pmatrix}
X_{1,1}&X_{1,2}&X_{1,3}\\
X_{2,1}&X_{2,2}&X_{2,3}\\
X_{3,1}&X_{3,2}&X_{3,3}\\
\end{pmatrix}.$$

It is enough to prove that $X_{1,2}=X_{1,3}=X_{2,3}=X_{3,3}=0$. To this end, if we consider Theorem 1.19 in \cite{BarrySimon}, the following inequalities hold

$$\left\|\begin{pmatrix}
X_{1,1}&X_{1,2}\\
X_{2,1}&X_{2,2}
\end{pmatrix}\right\|_1+\left\|X_{3,3}\right\|_1\leq \left\|X\right\|_1$$
and
$$\left\|X_{1,1}\right\|_1+\left\|X_{2,2}\right\|_1\leq \left\|\begin{pmatrix}
X_{1,1}&X_{1,2}\\
X_{2,1}&X_{2,2}
\end{pmatrix}\right\|_1.$$

Suppose that $\left\|X_{3,3}\right\|_1\neq 0$, then
$$1=\text{tr}(X(C+D_1))=\text{tr}(X_{1,1})-\text{tr}(X_{2,2})+\text{tr}(X_{3,3}(C+D_1)_{3,3})$$
$$<\left\|X_{1,1}\right\|_1+\left\|X_{2,2}\right\|_1+\left\|X_{3,3}\right\|_1\leq \left\|\begin{pmatrix}
X_{1,1}&X_{1,2}\\
X_{2,1}&X_{2,2}
\end{pmatrix}\right\|_1+\left\|X_{3,3}\right\|_1\leq \left\|X\right\|_1\leq 1,$$
which is a contradiction. Then, $X_{3,3}=0$.

It also follows that
$$\left.
\begin{array}{c l r l}
\text{tr}(X_{1,1})=\left\|X_{1,1}\right\|_1\\
\text{tr}(-X_{2,2})=\left\|-X_{2,2}\right\|_1 \\
\end{array}
\right\}\Rightarrow X_{1,1}\geq 0\ \wedge -X_{2,2}\geq 0.$$

On the other hand,
$$1=\text{tr}(X(C+D_1))=\left\|X_{1,1}\right\|_1+\left\|-X_{2,2}\right\|_1\leq \left\|X(C+D_1)\right\|_1\leq \left\|X\right\|_1\left\|C+D_1\right\|\leq 1.$$
Therefore,
$$\text{tr}(X(C+D_1))=\left\|X(C+D_1)\right\|_1.$$
Then $X(C+D_1)\geq 0$, which implies that
$$\left\{
\begin{array}{c l r l}
X_{3,1}(C+D_1)_{3,3}=X_{1,3}^*(C+D_1)_{3,3}=X_{3,1}\ \Leftrightarrow X_{3,1}=X_{1,3}^*=0\\
X_{3,2}(C+D_1)_{3,3}=X_{2,3}^*(C+D_1)_{3,3}=X_{3,2}\ \Leftrightarrow X_{3,2}=X_{2,3}^*=0\\
\end{array}
\right..$$

Analogously, we can deduce that 
$$\text{tr}\begin{pmatrix}
X_{1,1}&X_{1,2}\\
-X_{2,1}&-X_{2,2}\\
\end{pmatrix}=\left\|\begin{pmatrix}
X_{1,1}&X_{1,2}\\
-X_{2,1}&-X_{2,2}\\
\end{pmatrix}\right\|_1.$$
Then $\begin{pmatrix}
X_{1,1}&X_{1,2}\\
-X_{2,1}&-X_{2,2}\\
\end{pmatrix}\geq 0$ and
$-X_{2,1}=X_{1,2}^*=X_{2,1}=0.$
Therefore, 
$$X=\begin{pmatrix}
X_{1,1}&0&0\\
0&X_{2,2}&0\\
0&0&0\\
\end{pmatrix}$$

and this operator commutes with $C+D_1$. Also, 
$$X^+=E_+X_{1,1}E_+\ \Longrightarrow E_+X^+=X^+\ {\rm and}\ X^-=E_-X_{2,3}E_-\ \Longrightarrow E_-X^-=X^-.$$

(2) $\Rightarrow$ (3) Let $X\in \bb_1(\hh)^h$, $X\neq 0$ such that $\text{diag}(X)=0$, $\text{tr}(CX)=\left\|X\right\|_1$ and $E_+X^+=X^+\ ,\ E_-X^-=X^-$. Let $D\in \D(\kh)$ and define numbers $m$ and $M$ as
\begin{equation}
m=\min_{y\in R(E_+)}\dfrac{\left\langle Dy,y\right\rangle}{\left\|y\right\|^2}\ \ ,\ \ M=\max_{z\in R(E_-)}\dfrac{\left\langle Dz,z\right\rangle}{\left\|z\right\|^2}. \label{minimax}
\end{equation}

Observe that $ran(E_+),ran(E_-)<\infty$, so the minimum and maximum, respectively, are always attained. We claim that 
\begin{equation}
\text{tr}\left(\dfrac{X^+}{\left\|X^+\right\|_1}D\right)\geq m.\label{desigualdad}
\end{equation}

In order to prove it observe that $X^+=E_+X^+$ and note that 
$$\text{tr}\left(\dfrac{X^+}{\left\|X^+\right\|_1}D\right)=\text{tr}\left(\dfrac{E_+X^+E_+}{\left\|X^+\right\|_1}D\right)=\text{tr}\left(\dfrac{X^+}{\left\|X^+\right\|_1}E_+DE_+\right).$$
Therefore, inequality (\ref{desigualdad}) is equivalent to
$$\text{tr}\left[\dfrac{X^+}{\left\|X^+\right\|_1}\left(E_+DE_+-mE_+\right)\right]\geq 0,$$
since $\dfrac{X^+}{\left\|X^+\right\|_1}\geq 0$. Then, if we prove that $E_+DE_+-mE_+\geq 0$ we obtain (\ref{desigualdad}). Let $h\in \hh$: 

$$\left\langle E_+DE_+h,h\right\rangle=\left\langle DE_+h,E_+h\right\rangle=\left\langle Dy,y\right\rangle\geq m\left\|y\right\|^2,$$
with $E_+h=y\in R(E_+)$. Then, $\underbrace{\left\langle Dy,y\right\rangle}_{<\infty}-\underbrace{m\left\langle y,y\right\rangle}_{<\infty}\geq 0$, for all $y\in R(E_+)$. Finally, since $y=E_+h$, we have
$$\left\langle \left(DE_+-mE_+\right)h,E_+h\right\rangle\geq 0\ \Leftrightarrow \left\langle \left(E_+DE_+-mE_+\right)h,h\right\rangle\geq 0.$$
Analogously, it can be proved that $\text{tr}\left(\dfrac{X^-}{\left\|X^-\right\|_1}D\right)\leq M$.

On the other hand, the condition $\text{diag}(X)=0$ with $X\neq 0$ forces that $\text{diag}(X^+)=\text{diag}(X^-)\neq 0$, and since $X^+,X^-\geq 0$ we have 
$$\left\|X^+\right\|_1=\left\|\text{diag}(X^+)\right\|_1=\left\|\text{diag}(X^-)\right\|_1=\left\|X^-\right\|_1$$
and 
$$\text{tr}(X^+D)=\text{tr}(X^-D).$$
 
Therefore, there exist $y_0\in R(E_+)$ and $z_0\in R(E_-)$ such that $\left\|y_0\right\|=\left\|z_0\right\|=1$ and
$$\left\langle Dy_0,y_0\right\rangle=m\leq \text{tr}\left(\dfrac{X^+}{\left\|X^+\right\|_1}D\right)=\text{tr}\left(\dfrac{X^-}{\left\|X^-\right\|_1}D\right)\leq M=\left\langle Dz_0,z_0\right\rangle.$$

(3) $\Rightarrow$ (2) For this part we follow the main ideas used in the proof of Theorem 2 in \cite{varela}: take the function $\Phi(X)=\text{diag}(X)$ defined in section \ref{preliminares} and the following sets
$$\mathcal{A}=\left\{Y\in \bb_1(\hh)^h:E_+Y=Y\geq 0\ ,\text{tr}(Y)=1\right\}\ {\rm and}\ \mathcal{B}=\left\{Z\in \bb_1(\hh)^h:E_-Z=Z\geq 0\ ,\ \text{tr}(Z)=1\right\}.$$
Since $ran(E_+)<\infty$ (and $ran(E_-)<\infty$), every $Y\in \mathcal{A}$ (and every $Z\in \mathcal{B}$) is an Hermitian operator between finite fixed dimensional spaces. Then, all norms restricted to those spaces are equivalent. Thus, we can consider that $\Phi(\mathcal{A})$ and $\Phi(\mathcal{B})$ are compact subsets of $l^2(\R)$ for every norm (and of course, they are convex also).

Assume the non existence of $X$ satisfying (2). This implies that $\Phi(\mathcal{A})\cap \Phi(\mathcal{B})=\emptyset$. Since $\Phi(\mathcal{A})$ and $\Phi(\mathcal{B})$ are compact and convex sets of $l^2(\R)$ considering the euclidean norm, there exist $a,b\in \R$ and a functional $\rho$ defined for every $x\in \hh$ such that $\rho(x)=\sum_{i=1}^{\infty}x_id_i$, with $d=(d_i)_{i\in \N}\in c_0$, such that
$$\rho(y)\geq a>b\geq \rho(z),$$
for each $y\in \Phi(\mathcal{A})$ and $z\in \Phi(\mathcal{B})$. Then 
$$\left\langle \Phi(Y),d\right\rangle\geq a>b\geq \left\langle \Phi(Z),d\right\rangle\ \Rightarrow \min_{Y\in \mathcal{A}}\left\langle \Phi(Y),d\right\rangle> \max_{Z\in \mathcal{B}}\left\langle \Phi(Z),d\right\rangle,$$
and this can not occur  because if $D=\text{Diag}(d)\in \D(\kh)$, then 
$$\min_{Y\in \mathcal{A}}\left\langle \Phi(Y),d\right\rangle=m\ {\rm and}\ \max_{Z\in \mathcal{B}}\left\langle \Phi(Z),d\right\rangle=M,$$
with $m$ and $M$ defined in (\ref{minimax}). Therefore $M<m$ and this fact contradicts condition (3).

\end{proof}

\begin{rem}
The operator $X$ in statement (2) of Theorem \ref{teoprincipal} has finite rank. Moreover, $X$ can be described as a finite diagonal block operator in the base of eigenvectors of the minimal compact operator $C+D_1$.
\end{rem}

\begin{rem}
Let $C\in \kh$ and suppose that there exists an operator $X$ which satisfies the conditions of statement (\textit{2}) of Theorem \ref{teoprincipal}.  Then, we can define $\Psi:\kh\to\R$, given by $\Psi(\cdot)=\text{tr}(X\cdot)$, such that 
\begin{enumerate}
\item $\left\|\Psi\right\|\leq 1,$
\item $\Psi(C)=\text{tr}(XC)=\left\|\left[C\right]\right\|,$
\item $\Psi(D)=0\ \forall \ D\in \D(\bh),$
\end{enumerate}
and $\Psi$ acts as a functional witness of the minimality of $C$ (see \cite{rieffel}).
\end{rem}

If we take $v,w\in \hh$, we can write $v=\sum_{i=1}^{\infty}v^ie_i$ and $w=\sum_{i=1}^{\infty}w^ie_i$ with $v^i,w^i\in \C$ for all $i\in \N$. Then, we denote with $v\circ w$ the vector in $\hh$ defined by
$$v\circ w=\left(v^1w^1,v^2w^2,v^3w^3,...\right)\in \hh.$$

The proof of the following corollary is the analogue of that of Corollary 3 in \cite{varela}, considering the special treatment for compact operators instead of matrices.

\begin{cor}
Let $C\in \kh$, $C\neq 0$, such that $\lambda_{max}(C)+\lambda_{min}(C)=0$. Then, the following statements are equivalent:
\be
\item $C$ is minimal (as defined in the Section \ref{preliminares}). 
\item There exist $\left\{v_i\right\}_{i=1}^{r}\subset ran(E_+)$ and $\left\{v_j\right\}_{j=r+1}^{r+s}\subset ran(E_-)$, orthonormal sets such that
$$\text{co}\left(\left\{v_i\circ\overline{v_i}\right\}_{i=1}^{r}\right)\cap \text{co}\left(\left\{v_j\circ\overline{v_j}\right\}_{i=r+1}^{r+s}\right)\neq 0.$$ 
\ee
\end{cor}
Here $\text{co}\left(\left\{w_k\right\}_{k=n_0}^{n_1}\right)$ denotes the convex hull of the space generated by the finite family of vectors $\left\{w_k\right\}_{k=n_0}^{n_1}\subset\hh$, and if $w_k= ({w_k^1},{w_k^2},{w_k^3},...)$ in the canonical or fixed base chosen in $H$ (see Section \ref{preliminares}), then we denote with 
$\overline{w_k}= (\overline{w_k^1},\overline{w_k^2},\overline{w_k^3},...)\in\hh$.

\end{document}